\documentclass[a4paper, english,reqno]{amsart}
\usepackage{amsmath} % utilities for mathematics
\usepackage{amsthm} % utilities for theorem environment
\usepackage{amssymb} % loads mathematical fonts

\usepackage{amscd} % utilities for commutative diagrams

\usepackage{array} % core package
\usepackage[backref=page,linktocpage]{hyperref} % permits navigating on the pdf
\usepackage{caption} %
\usepackage{graphics,graphicx} % permits putting images
\usepackage{tikz,tikz-cd} % tool for diagrams
\usepackage{enumerate} % permits enumerating using letters or other symbols

\usepackage{xcolor}

\makeatletter
\@namedef{subjclassname@2020}{\textup{2020} Mathematics Subject Classification}
\makeatother

\hypersetup{
    colorlinks = true,
    linkbordercolor = {white},
    linkcolor = {blue},
    anchorcolor = {black},
    citecolor = {blue},
    filecolor = {cyan},
    menucolor = {red},
    runcolor = {cyan},
    urlcolor = {black}
}

\usetikzlibrary{automata}

\theoremstyle{plain}
\newtheorem{theorem}{Theorem}[section]

\theoremstyle{definition}
\newtheorem{definition}[theorem]{Definition}

\newtheorem{example}[theorem]{Example}

\theoremstyle{plain}

\newtheorem{corollary}[theorem]{Corollary}
\newtheorem{lemma}[theorem]{Lemma}

\newtheorem*{theorem-no-label}{Theorem}

\newcommand{\U}{\mathsf{U}}
\newcommand{\B}{\mathsf{B}}
\renewcommand{\H}{\mathrm{H}}

\newcommand{\A}{\mathrm{A}}

\title{Chow polynomials of uniform matroids \\are real-rooted}

\author[P.~Br\"and\'en]{Petter Br\"and\'en}
\address{Department of Mathematics, KTH Royal Institute of Technology, SE-100 44 Stockholm,
Sweden}
\email{pbranden@kth.se, lvecchi@kth.se}

\author[L.~Vecchi]{Lorenzo Vecchi}
\thanks{PB is a Wallenberg Scholar
  supported by the Knut and Alice Wallenberg
  Foundation, and the G\"oran Gustafsson foundation.}
\begin{document}

\begin{abstract}
  June Huh and Matthew Stevens conjectured that the Hilbert--Poincar\'e series of the Chow ring of any matroid is a polynomial with only real zeros. 
  We prove this conjecture for the class of uniform matroids. We also prove that the Chow polynomial and the augmented Chow polynomial of any maximal ranked poset have only real zeros. 
\end{abstract}

\maketitle

\section{Introduction} \thispagestyle{empty}
The introduction of combinatorial Hodge theory led not only to the solution of long-standing conjectures on positivity in matroid theory \cite{adiprasito-huh-katz,braden-huh-matherne-proudfoot-wang}, but also opened up new lines of research, and led to new positivity questions. For example, the Chow ring and augmented Chow ring of matroids---which played key roles in solving these conjectures---have \emph{Hilbert--Poincaré series} that are now conjectured to be real-rooted polynomials \cite[{Conjecture~4.1.3 and 4.3.3}]{stevens-bachelor}.

In \cite{ferroni-matherne-vecchi}, Ferroni, Matherne and the second author generalized this construction and introduced new classes of polynomials associated to bounded and graded posets, called the \emph{Chow polynomial} and \emph{augmented Chow polynomial} of a poset $P$, denoted here by $\H_P(t)$ and $G_P(t)$, respectively. When $P$ is a geometric lattice, these correspond to the Hilbert--Poincaré series of the Chow ring and augmented Chow ring of the corresponding matroid. In \cite[Conjecture~1.5]{ferroni-matherne-vecchi} the authors conjectured that all zeros of these polynomials are real for each \emph{Cohen--Macaulay poset}, and in support of this they proved the $\gamma$-positivity \cite[Theorem~1.4]{ferroni-matherne-vecchi} of these polynomials.  Very little progress has so far been made towards the resolution of these conjectures, showing the need for new methods to approach such problems. 
 In this note we settle the real-rootedness conjecture for Chow rings of uniform geometric lattices. These polynomials were recently studied in \cite{coron, hoster}.
\begin{theorem}\label{chowth}
    Let $\U_{k,n}$ be the uniform geometric lattice of rank $k$ over $n$ elements. Then all zeros of $\H_{\U_{k,n}}(t)$ are real. 
\end{theorem}
While the real-rootedness of $G_{\U_{k,n}}(t)$ has already been settled in \cite[Theorem~1.10]{ferroni-matherne-stevens-vecchi}, we present a new and independent proof of this fact here in analogy with the proof of Theorem \ref{chowth}. We also prove real-rootedness for two other classes of posets. Let $P^*$ denote the dual of a poset $P$, i.e., the poset where $x \leq^* y$ if and only if $y \leq x$ in $P$. Moreover, define the \emph{truncation} of $P$ to be the poset $\tau(P)$ where we remove all the coatoms of $P$ and the \emph{dual truncation} to be the poset $\sigma(P)$ where we remove all the atoms.
\begin{theorem}\label{thm:truncation and dual truncation}
    Let $P$ be a bounded and graded poset. Then,
    \[
    \H_P(t) = G_{\sigma(P)}(t) = G_{\tau(P^*)}(t).
    \]
\end{theorem}
As a direct consequence, we conclude the real-rootedness of the dual of uniform geometric lattices.
\begin{corollary}\label{cor:dual uniform}
    Let $\U_{k,n}$ be the uniform geometric lattice of rank $k$ over $n$ elements. Then, $\H_{\left(\U_{k,n}\right)^*}(t)$ and $G_{\left(\U_{k,n}\right)^*}(t)$ have only real zeros.
\end{corollary}
Observe that the dual operation we are considering is on the poset, not on the matroid. In particular, the poset we obtain is not the lattice of flats of a matroid. Lastly, we consider \emph{maximal ranked posets}, i.e., graded posets in which all possible order relations between elements of different ranks are present.
\begin{theorem}\label{maxrank}
    Let $P$ be a maximal ranked poset. Then $\H_P(t)$ and $G_P(t)$ have only real zeros. 
\end{theorem}
\section{Background}

\subsection{Posets}
For undefined terminology on partially ordered sets (posets), we refer to \cite[Chapter 3]{stanley-ec1}.
In a bounded and graded poset with rank function $\rho$, we call $P_m = \{x\in P \mid \rho(x) = m\}$  the $m$-th rank level, and we call the elements of $P_1$ the atoms and the elements of $P_{\rho(P)-1}$ the coatoms of $P$. We also write $\widehat{0}$ and $\widehat{1}$ for the least and largest element, respectively. The \emph{truncation} $\tau(P)$ of $P$ is the poset obtained by removing all coatoms of $P$. Dually, we define the \emph{dual truncation} $\sigma(P)$ to be the poset obtained by removing all the atoms of $P$. Lastly, the \emph{augmentation} $\operatorname{aug}(P)$ is constructed by introducing a new element and declaring it to be the new least element of the poset. 

\subsubsection{Uniform geometric lattices}
We remark that geometric lattices are in bijection with the class of simple matroids. However, since all our tools can be described strictly in terms of poset-theoretic constructions, we prefer to define these objects only as posets, with the knowledge that geometric lattices carry with them additional structures. 
\begin{definition}
    Let $n$ be a nonnegative integer. The Boolean algebra on $n$ elements $\B_n$ is the partially ordered set defined by ordering the power set $2^{[n]}$ by inclusion.
\end{definition}

\begin{definition}
    The uniform geometric lattice $\U_{k,n}$ is the poset obtained by truncating $(n-k)$ times the Boolean algebra $\B_n$, that is,
    \[
    \U_{k,n} := \tau^{n-k}\B_n.
    \]
\end{definition}
Observe that the Boolean algebra is a uniform geometric lattice, as $\B_n = \U_{n,n}$.

\subsubsection{Maximal ranked posets}
Fix a sequence of positive integers $c = (c_i)_{i\geq 1}$. Consider the infinite family $P^{c}_{n}$ constructed by declaring the rank level $P_i$ to have $c_i$ elements for $1 \leq i \leq n$, adding a minimal element $\widehat{0}$ and a maximal element $\widehat{1}$ and all the possible order relations among rank levels.
\begin{example}
Let $(c_i)_i=(3,3,\ldots)$. The poset $P^c_2$ is depicted in Figure \ref{fig:Pc2}. 
\begin{figure}[ht]
    \centering
	\begin{tikzpicture}  
	[scale=0.7,auto=center,every node/.style={circle,scale=0.8, fill=black, inner sep=2.7pt}] 
	\tikzstyle{edges} = [thick];
	
	\node[] (a1) at (0,0) {};  
	\node[] (a2) at (-1,1)  {};  
	\node[] (a3) at (0,1) {};
	\node[] (a4) at (1,1) {};
	\node[] (a5) at (-1,2)  {};  
	\node[] (a6) at (0,2)  {};  
	\node[] (a7) at (1,2)  {};  
    \node[] (a8) at (0,3) {};
	
	\draw[edges] (a1) -- (a2);  
	\draw[edges] (a1) -- (a3);  
	\draw[edges] (a1) -- (a4);
	\draw[edges] (a2) -- (a5);
        \draw[edges] (a2) -- (a6);
	\draw[edges] (a2) -- (a7);
	\draw[edges] (a3) -- (a5);
	\draw[edges] (a3) -- (a6);
	\draw[edges] (a3) -- (a7);
	\draw[edges] (a4) -- (a5);
	\draw[edges] (a4) -- (a6);
	\draw[edges] (a4) -- (a7);
    \draw[edges] (a5) -- (a8);
	\draw[edges] (a6) -- (a8);
	\draw[edges] (a7) -- (a8);
	\end{tikzpicture}\caption{A maximal ranked poset.}\label{fig:Pc2}
\end{figure}
\end{example}
Clearly, the class of maximal ranked posets is closed under the operations of truncation, dual truncation and augmentation.

\subsection{Derangements and Eulerian polynomials}
We recall the definitions of two families of polynomials that are important when dealing with Hilbert series of Chow rings and augmented Chow rings of uniform and Boolean geometric lattices.

\subsubsection{Eulerian polynomials} 

One of the most pervasive objects in enumerative combinatorics is the family of  Eulerian polynomials. Given a permutation $\pi\in\mathfrak{S}_n$, $n\geq 1$, written in one-line notation as $\pi=\pi_1\pi_2 \cdots\pi_n$, the number of \emph{descents} of $\pi$ is defined as the cardinality of the set $\{i\in [n-1]:\pi_i>\pi_{i+1}\}$, and is denoted by $\operatorname{des}(\pi)$. For $n\geq 1$, we define the $n$-th \emph{Eulerian polynomial} as follows:
    \[ A_n(t) := \sum_{\pi\in\mathfrak{S}_n} t^{\operatorname{des}(\pi)}.\] 
We note that this differs by a factor of $t$ from the definition in \cite[p.~33]{stanley-ec1}. Furthermore, we define $A_0(t)=1$. 

The polynomials $A_n(t)$ are palindromic and $\deg A_n(t) = n-1$. It is a classical result attributed to Frobenius that these polynomials are real-rooted (for a proof, see \cite[Example 7.3]{branden}). The coefficients of the Eulerian polynomials admit several combinatorial interpretations, many of which can be found in \cite[Chapter~1]{stanley-ec1}.

\subsubsection{Derangement polynomials}

A permutation $\pi\in\mathfrak{S}_n$ is said to be a \emph{derangement} if $\pi_i\neq i$ for all $i$, i.e., if $\pi$ has no fixed points. The set of all derangements on $n$ elements is usually denoted by $\mathfrak{D}_n$. For each $n\geq 1$, the $n$-th \emph{derangement polynomial}, denoted $d_n(t)$, is defined by
    \[ d_n(t) := \sum_{\pi\in\mathfrak{D}_n} t^{\operatorname{exc}(\pi)}.\]
where $\operatorname{exc}(\pi):=|\{i\in [n]:\pi_i>i\}|$ denotes the number of \emph{excedances} of $\pi$. Also, $d_0(t):=1$.

Notice that $\deg d_n(t) = n-1$ for each $n\geq 1$. With only the exception of $d_0(t)=1$, the polynomial $d_n(t)$ is a multiple of $t$ and is symmetric (palindromic) with center of symmetry ${n}/{2}$. Lastly, derangement polynomials are known to be real-rooted (see e.g. \cite[Section 3.2]{branden-solus}).

\subsection{Characteristic Chow polynomials of graded bounded posets}
In this section we recall the definition of the main classes of polynomials of interest. We will be focusing on what in \cite{ferroni-matherne-vecchi} is called the \emph{characteristic Chow function}, even though we remark that these definitions can be given with respect to any \emph{$P$-kernel}. Since we are not concerned with different types of Chow functions, we will call these for simplicity just \emph{Chow functions}. For further details we refer to \cite[Section~3 and 4]{ferroni-matherne-vecchi}. 

Let $P$ be a bounded and graded poset. Recall that the \emph{M\"obius function} of $P$ $\mu\colon P \times P\to\mathbb{Z}$ is defined as
\[
\mu(x,y) = \mu_{xy} :=
\begin{cases}
    0 & \text{if $x \not< y$}\\
    1 & \text{if $x=y$}\\
    -\sum_{x\leq z < y}\mu_{xz} &\text{otherwise}.
\end{cases}
\]
The \emph{characteristic function} of $P$  is defined as the element of the incidence algebra $\chi \in \mathcal{I}(P)$ that associates to an interval $[x,y]$ the polynomial
\[
\chi_{[x,y]}(t) = \sum_{x\leq z \leq y}\mu_{xz}t^{\rho(y) - \rho(z)}.
\]
It is not hard to show that whenever the interval $[x,y]$ is not trivial, then $\chi_{[x,y]}(t)$ is divisible by $t-1$. The reduced characteristic function is defined as
\[
\overline{\chi}_{[x,y]}(t) = \begin{cases}
    -1 & \text{if $x=y$}\\
    \frac{\chi_{[x,y]}(t)}{t-1} & \text{otherwise}
\end{cases},
\]
and the \emph{characteristic Chow function} is then defined as $\H = (-\overline{\chi})^{-1}$, i.e., the inverse (up to a sign) of the reduced characteristic function in the incidence algebra. More explicitly
\[
\H_{[x,y]}(t) := \sum_{x<z\leq y}\overline{\chi}_{[x,z]}(t)\H_{[z,y]}(t).
\]

The \emph{Chow polynomial} of the poset $P$ is then defined by  $\H_P(t) := \H_{[\widehat{0},\widehat{1}]}(t)$. For every bounded and graded poset $P$, $\H_P(t)$ is palindromic \cite[Proposition~3.6(ii)]{ferroni-matherne-vecchi} of degree $\rho(P)-1$, and its coefficients form a non-negative and unimodal sequence \cite[Theorem~1.1]{ferroni-matherne-vecchi}. In addition to that, it is $\gamma$-positive for Cohen--Macaulay posets \cite[Theorem~1.4]{ferroni-matherne-vecchi} and is conjectured to be real-rooted for every Cohen--Macaulay poset \cite[Conjecture~1.5]{ferroni-matherne-vecchi}. For geometric lattices, the real-rootedness conjecture was first formulated in \cite[Conjecture~4.1.3]{stevens-bachelor}. The following theorem suggests that there is a rich enumerative structure to be revealed in the theory of Chow polynomials.

\begin{theorem}[{\cite[Theorem~5.1]{hameister-rao-simpson}}]
    Let $n\geq 2$. Then,
    \[
    A_n(t) = \H_{B_n}(t) \qquad \text{and} \qquad d_n(t) = t\H_{\U_{n-1,n}}(t).
    \]
\end{theorem}

One could use the given definition of $\H_P(t)$ as a recursive formula to compute the Chow polynomial of an interval by knowing all Chow polynomials on smaller sub-intervals. One of the main issues with this formula is that the recursion involves two different types of polynomials, namely $\overline{\chi}$ and $\H$. We now recall two formulas that bypass this issue. 

The following recursion lets us compute the Chow polynomial of $P$ in terms of its truncations. 
\begin{theorem}[{\cite[Proposition~4.12]{ferroni-matherne-vecchi}}]\label{thm:truncation formula}
Let $P$ be a bounded and graded poset. Then
    \[
    \H_P(t) = 1 + t\sum_{\substack{x\in P \\ \rho(x) \geq 2}}\H_{\tau([\widehat{0},x])}(t).
    \]
\end{theorem}
The formula was first proven to hold on geometric lattices in \cite[Corollary~3.5]{larson}.

The following identity will be used later in the paper
\begin{lemma}\label{lemma: truncation}
Let $P$ be a bounded and graded poset of rank $r$. Then
    \[
    \H_P(t) = (1+t)\H_{\tau(P)}(t) - t \H_{\tau^2(P)}(t) + t\sum_{\rho(x) = r-1}\H_{\tau([\widehat{0},x])}(t).
    \]
\end{lemma}
\begin{proof}
    We use Theorem \ref{thm:truncation formula} on $P$ and $\tau(P)$. 
    \begin{align*}
        \H_P(t) &= 1 + t\sum_{2\leq \rho(x)\leq r-1}\H_{\tau([\widehat{0},x])}(t) + t\H_{\tau(P)}(t),\\
        \H_{\tau(P)}(t) &= 1 + t\sum_{2\leq \rho(x)\leq r-2}\H_{\tau([\widehat{0},x])}(t) + t\H_{\tau^2(P)}(t).\\
    \end{align*}
    The result follows after subtracting the second equation from the first.
\end{proof}

The second formula is a specialization of \cite[Theorem~3.9]{ferroni-matherne-vecchi} in the case of characteristic Chow polynomials.
\begin{theorem}\label{thm: numerical-canonical-decomposition}
Let $P$ be a bounded and graded poset. Then
    \[
    \H_P(t) = \frac{t^{\rho(P)}-1}{t-1} + t\sum_{\widehat{0}<x<\widehat{1}} \frac{t^{\rho(x)-1}-1}{t-1}\H_{[x,\widehat{1}]}(t).
    \]
\end{theorem}
This is also referred to as one of the numerical canonical decompositions, see the discussion in \cite[Section~4.4]{ferroni-matherne-vecchi} for an explanation regarding the algebro-geometric motivation behind this name.

A closely related family of polynomials is the one of \emph{(left) augmented Chow polynomials}. Given a bounded and graded poset $P$ we define for every $x\leq y$
\[
G_{[x,y]}(t) := \sum_{x\leq z \leq y}t^{\rho(z)-\rho(x)}\H_{[z,y]}(t),
\]
and call $G_P(t) := G_{[\widehat{0},\widehat{1}]}(t)$ the left augmented Chow polynomial of $P$. These polynomials were first introduced in the combinatorial Hodge theory of geometric lattices \cite{semismall}. However, the following result lets us deal with augmented Chow polynomials as Chow polynomials. 
\begin{theorem}[{\cite[Corollary~4.6]{ferroni-matherne-vecchi}}]\label{thm: G is H aug}
Let $P$ be a bounded and graded poset. Then
\[
G_P(t) = \H_{\operatorname{aug}(P)}(t).
\]
\end{theorem}
Since augmenting a poset preserves Cohen--Macaulayness, augmented Chow polynomials are also conjectured to be real rooted for Cohen--Macaulay posets. The conjecture for matroids can first be found in \cite[Conjecture~4.3.3]{stevens-bachelor} and was settled in the uniform case in \cite[Theorem~1.10]{ferroni-matherne-stevens-vecchi}.

\section{Uniform geometric lattices}
The goal of this section is to prove Theorem \ref{chowth}. We write $f_{n,k}(t) := \H_{\U_{k+1,n}}(t)$. The shift for $k$ is just for convenience as, for example, it ensures that the polynomial $f_{n,k}(t)$ has degree $k$. 

Suppose $f(t), g(t) \in \mathbb{R}[t]$ are real-rooted polynomials with positive leading coefficients, and that 
$$
\cdots \leq \alpha_3 \leq \alpha_2 \leq \alpha_1  \ \  \mbox{ and } \ \ \cdots \leq \beta_3 \leq \beta_2 \leq \beta_1 
$$
are the zeros of $f(t)$ and $g(t)$, respectively. We say that the zeros of $f(t)$ \emph{interlace} those of $g(t)$ if 
$$
 \cdots \leq \alpha_3 \leq \beta_3 \leq \alpha_2 \leq \beta_2 \leq \alpha_1 \leq \beta_1, 
$$
and we write $f(t) \prec g(t)$. 

\newtheorem*{thm:chowth}{Theorem~\ref{chowth}}
\begin{thm:chowth}
       Let $\U_{k,n}$ be the uniform geometric lattice of rank $k$ over $n$ elements. Then all zeros of $\H_{\U_{k,n}}(t)$ are real. 
\end{thm:chowth}
\begin{proof}
We employ Theorem \ref{thm:truncation formula}. In the uniform case the recursion becomes
\begin{equation}\label{rec1}
f_{n,k}(t) = tf_{n,k-1}(t) + D_{n,k}(t),
\end{equation}
where $D_{n,k}(t) = \sum_{j=0}^k \binom{n}{j}d_j(t).$
Observe that in the uniform case the truncation formula could also be deduced directly by rewriting the equation in \cite[Theorem~3.21]{ferroni-matherne-stevens-vecchi}. By exploiting the palindromicity \cite[Proposition~3.6 (ii)]{ferroni-matherne-stevens-vecchi} of the Chow polynomials, we rewrite \eqref{rec1} as

\begin{align}\label{rec2}
    f_{n,k}(t) &= t^kf_{n,k}(t^{-1}) = t^k\left(t^{-1} f_{n,k-1}(t^{-1}) +D_{n,k}(t^{-1}) \right) \nonumber \\
    &= f_{n,k-1}(t) + t^kD_{n,k}(t^{-1}).
\end{align}
By subtracting \eqref{rec2} from \eqref{rec1} we deduce
\[
(t-1)f_{n,k-1}(t) = t^kD_{n,k}(t^{-1}) - D_{n,k}(t).
\]
Let $D$ be the \emph{deranged map}, i.e., the linear map defined by $D(t^j) = d_j(t)$ in \cite[Section~3.2]{branden-solus}. Then $D_{n,k}(t) = D\left(\sum_{j=0}^k\binom{n}{j}t^j\right)$.
Since
\[
\sum_{j=0}^k\binom{n}{j}t^j = \sum_{i=0}^kh_it^i(1+t)^{k-i},
\]
where $h_i = \binom{n-k+i-1}{i} \geq 0$, we can use \cite[Corollary~3.7]{branden-solus} to deduce 
\[
D_{n,k}(t) \prec t^kD_{n,k}(t^{-1}).
\]
Obreshkoff's theorem (see e.g. \cite[Theorem~7.2]{branden}) says that any linear combination of two polynomials whose zeros are interlacing is real-rooted. Hence $t^kD_{n,k}(t^{-1}) - D_{n,k}(t)$ is real-rooted, and then so is $f_{n,k-1}(t)$.
\end{proof}

The analogous theorem for the augmented Chow polynomial was already proved in {\cite[Theorem~1.10]{ferroni-matherne-stevens-vecchi}}. We present an alternative proof in analogy with the proof of Theorem \ref{chowth}.
\begin{theorem}\label{thm: G real rooted}
    All zeros of the augmented Chow polynomial of a uniform geometric lattice $\U_{k,n}$ are real.
\end{theorem}
 
\begin{proof}
Let $g_{n,k}(t)$ denote the augmented Chow polynomial of the uniform geometric lattice $\U_{k,n}$. Again, using either \cite[Corollary~3.5]{larson} or \cite[Theorem~3.21]{ferroni-matherne-stevens-vecchi},
\[
g_{n,k}(t) = tg_{n,k-1}(t) + A_{n,k}(t),
\]
where $A_{n,k}(t) = \sum_{j=0}^{k-1}\binom{n}{j}A_j(t)$. Again, as in the proof of Theorem \ref{chowth},
\[
(t-1)g_{n,k-1}(t) = t^kA_{n,k}(t^{-1}) - A_{n,k}(t).
\]
Let $\mathcal{A}^\circ$ be the Eulerian transformation as defined in  \cite{branden-jochemko}, $\mathcal{A}^\circ(t^j) = A_j(t)$. Then  
\[
A_{n,k}(t) = \mathcal{A}^\circ\left( \sum_{j=0}^{k-1}\binom{n}{j}t^j \right).
\]
By \cite[Theorem~1.1]{athanasiadis-eulerian2}, we know that $\A_{n,k}(t)$ is real-rooted and
\[
t^kA_{n,k}(t^{-1}) \prec A_{n,k}(t).
\]
We conclude the real-rootedness of $g_{n,k-1}(t)$ as in the proof of Theorem \ref{chowth}.
\end{proof}

\section{Dual truncation}
We now investigate how Chow polynomials behave under the operation of dual truncation. Notice that this operation would not make sense on the class of geometric lattices, as the dual truncation of a geometric lattice is not necessarily a geometric lattice.

\begin{theorem}\label{thm: dual truncation}
Let $P$ be a bounded and graded poset. Then
    \[
    \H_P(t) = 1 + t\sum_{2\leq \rho(x)\leq \rho(P) - 1}\H_{[x,\widehat{1}]}(t) + t\H_{\sigma(P)}(t).
    \]
    Alternatively,
    \[
    \H_P(t) = (t+1)\H_{\sigma(P)}(t) - t\H_{\sigma^2(P)}(t) + t\sum_{\rho(x) = 2}\H_{[x,1]}(t).
    \]
\end{theorem}

 \begin{proof}
     To prove the first statement we use Theorem \ref{thm: numerical-canonical-decomposition} on $\H_{P}(t)$ and $\H_{\sigma(P)}(t)$. Taking the difference of the two equations implies that
     \[
     \H_P(t) - \H_{\sigma(P)}(t) = t^{\rho(P)-1} + t\sum_{2\leq\rho(x)\leq \rho(P)-1}t^{\rho(x)-2}\H_{[x,\widehat{1}]}(t).
     \]
     The result then follows directly by palindromicity. The second statement follows by applying the first to both $\H_P(t)$ and $\H_{\sigma(P)}(t)$ and again taking differences. 
 \end{proof}

The following result lets us compute explicitly the augmented Chow polynomial of a dual truncation.

\newtheorem*{truncation and dual truncation}{Theorem \ref{thm:truncation and dual truncation}}
\begin{truncation and dual truncation}
    Let $P$ be a bounded and graded poset. Then
    \[
    \H_P(t) = G_{\sigma(P)}(t) = G_{\tau(P^*)}(t).
    \]
\end{truncation and dual truncation}
\begin{proof}
    Recall that, by Theorem \ref{thm: G is H aug}, $G_Q = \H_{\operatorname{aug}(Q)}$. We prove the first equality using Theorem \ref{thm: dual truncation} on $\operatorname{aug}\sigma(P)$. This tells us that
    \[
    \H_{\operatorname{aug}\sigma(P)}(t) = (1+t) \H_{\sigma(P)}(t) - t\H_{\sigma^2(P)}(t) + t\sum_{\rho(x) = 2}\H_{[x,\widehat{1}]}(t),
    \]
    where we used  $\sigma\left( \operatorname{aug}\sigma(P) \right) = \sigma(P)$. The equality follows immediately from Theorem \ref{thm: dual truncation}. We prove the second equality by induction on the rank $r$ of $P$. If $r\leq 2$ there is nothing to prove. If $r\geq 3$, then we use Theorem \ref{thm: dual truncation} to write
    \begin{align*}
        \H_{\operatorname{aug}(\tau(P^*))}(t) &= (1+t)\H_{\tau(P^*)}(t) - t \H_{\sigma(\tau(P^*))}(t) + t \sum_{\rho(x) = r-1} \H_{\tau([x,\widehat{0}]^*)}(t) \\
        &= (1+t)\H_{\operatorname{aug}(\tau (\sigma(P)))}(t) - t \H_{\operatorname{aug}(\tau^2(\sigma(P)))}(t) + t \sum_{\rho(x) = r-1} \H_{\operatorname{aug}(\tau(\sigma[\widehat{0},x]))}(t) \\
        &= \H_{\operatorname{aug}(\sigma(P))}(t), 
    \end{align*}
    where in the second equality we apply the inductive hypothesis since each of those posets are of lower rank, and in the third equality we use Lemma \ref{lemma: truncation} and the fact that the truncation commutes with the augmentation, when $r\geq 2$.
\end{proof}

This result lets us conclude an even stronger result regarding the augmented Chow polynomial of the dual of a poset.\footnote{We thank Nicholas Proudfoot for pointing this to us.}

\begin{corollary}\label{cor:G dual P}
    Let $P$ be a bounded and graded poset. Then,
    \[
    G_P(t) = G_{P^*}(t).
    \]
\end{corollary}
\begin{proof}
    Let $Q = \operatorname{aug}(P)$. This means that $\sigma(Q) = P$ and
    \[
    P^* = \sigma(Q)^* = \tau(Q^*).
    \]
    By Theorem \ref{thm:truncation and dual truncation}, 
    \[
    G_{P^*}(t) = G_{\tau(Q^*)}(t) = G_{\sigma(Q)}(t) = G_P(t).
    \]
\end{proof}

We are now ready to prove Corollary \ref{cor:dual uniform}.

\newtheorem*{cor:dualuniform}{Corollary~\ref{cor:dual uniform}}
\begin{cor:dualuniform}
    Let $\U_{k,n}$ be the uniform geometric lattice of rank $k$ over $n$ elements. Then, $\H_{(\U_{k,n})^*}(t)$ and $G_{(\U_{k,n})^*}(t)$ have only real zeros.
\end{cor:dualuniform}
\begin{proof}
    By Theorem \ref{thm:truncation and dual truncation}, $\H_{(\U_{k,n})^*}(t) = G_{\U_{k-1,n}}(t)$, which is real-rooted by Theorem \ref{thm: G real rooted}. Moreover, by Corollary \ref{cor:G dual P}, $G_{(\U_{k,n})^*}(t) = G_{\U_{k,n}}(t)$, hence this polynomial is also real-rooted. 
\end{proof}

\section{Maximal ranked posets}
In this section we prove Theorem \ref{maxrank}, which establishes real-rootedness for the Chow polynomials of maximal ranked posets. One nice feature of this class of posets is that the lower ideal generated by an element of rank $k+1$ in $P^c_{n}$ is isomorphic to $P^c_{k}$ and the truncation of $P^c_{n}$ is isomorphic to $P^c_{n-1}$. Using this, we can specialize Lemma \ref{lemma: truncation} and compute $f_{n}(t) := \H_{P^c_{n}}(t)$ as
\[
f_n(t) = (1+t)f_{n-1}(t) + t(c_n-1)f_{n-2}(t), \ \ n \geq 2,
\]
with the initial conditions $f_0(t)=0$ and $f_1(t)=1$. Observe that, as defined, $P^c_{n}$ has rank $n+1$ and therefore $\deg f_n = n$.
\newtheorem*{thm:maxrank}{Theorem~\ref{maxrank}}
\begin{thm:maxrank}
    Let $P$ be a maximal ranked poset. Then $\H_P(t)$ and $G_P(t)$ have only real zeros.
\end{thm:maxrank}

\begin{proof}
We use the notation established above and show by induction on $n$ that ${f_n(t)\prec f_{n+1}(t)}$ for every $n \geq 1$. Since $f_1(t)=1+t$ and $f_2(t)= 1+(1+c_3)t + t^2$, where $c_3\geq 1$, the case when $n=1$ follows. Assume 
$f_{n-1}(t) \prec f_n(t)$, where $n \geq 2$. Then all zeros of $f_{n-1}(t)$ and $f_n(t)$ are nonpositive since the coefficients are nonnegative. Hence $f_n(t) \prec  t f_{n-1}(t)$ and $f_n(t) \prec (1+t) f_n(t)$.  Then 
$$
f_n(t) \prec (1+t) f_n(t) + (c_{n+1}-1) t f_{n-1}(t)=f_{n+1}(t),
$$
since the set of polynomials that are interlaced by a specific polynomial is a convex cone, see e.g. \cite[Lemma 2.6]{borcea-branden-london}. 
The result for $G_P(t)$ follows from observing that if $P$ is a maximal ranked poset, then so is $\operatorname{aug}(P)$.
\end{proof}

\section*{Acknowledgements}
The authors wish to thank Nicholas Proudfoot for his comments.

\bibliographystyle{amsalpha}
\bibliography{bibliography}

\end{document}